\documentclass[a4paper]{article}

\usepackage[english]{babel}

\usepackage[utf8]{inputenc}
\usepackage[utf8]{inputenc}
\usepackage{amsmath}
\usepackage{graphicx}
\usepackage{amssymb}
\usepackage{amsthm}
\usepackage{tikz-cd}
\usepackage{mathrsfs}
\usepackage[colorinlistoftodos]{todonotes}
\usepackage{enumitem}
\usepackage{yfonts}
\usepackage{ dsfont }
\usepackage{MnSymbol}
\usepackage{slashed}

\title{Dirichlet problem for maximal graphs of higher codimension}

\author{Yang Li}

\date{\today}
\newtheorem{thm}{Theorem}[section]
\newtheorem{lem}[thm]{Lemma}
\newtheorem{prop}[thm]{Proposition}

\theoremstyle{definition}

\newtheorem{rmk}{Remark}

\newtheorem*{Notation}{Notation}
\newtheorem*{Acknowledgement}{Acknowledgement}

\newcommand{\ie}{\emph{i.e.} }
\newcommand{\cf}{\emph{cf.} }

\newcommand{\R}{\mathbb{R}}

\newcommand{\Z}{\mathbb{Z}}

\newcommand{\norm}[1]{\left\lVert#1\right\rVert}
\newcommand{\Lap}{\Delta}

\DeclareMathOperator{\Tr}{Tr}

\begin{document}
	\maketitle

\begin{abstract}
We consider the Dirichlet boundary value problem for graphical maximal submanifolds inside Lorentzian type ambient spaces, and obtain general existence and uniqueness results which apply to any codimension. 
\end{abstract}

\section{Introduction}

Let $\R^{n,m}=\R^n_x \oplus \R^m_y$ denote the vector space endowed with a Lorentzian type bilinear form $\langle (x, y), (x, y)\rangle= |x|^2-|y|^2=x\cdot x-y\cdot y$ of signature $(n,m)$. Following conventions in general relativity, a vector $v\in \R^{n,m}$ is called spacelike (resp. null/timelike) if $\langle v, v \rangle$ is positive (resp. zero/negative).
A submanifold $\Sigma$ of dimension $n$ is called \textbf{spacelike} if the induced metric on its tangent spaces has the Riemannian signature; if morever the mean curvature vector vanishes identically on $\Sigma$, then $\Sigma$ is called a \textbf{maximal submanifold}. The goal of this paper is to study the Dirichlet problem for graphical solutions to the maximal submanifold equation (a.k.a `\textbf{maximal graphs}').

To set up, we represent $\Sigma$ as the graph of a smooth function $\vec{u}=(u^1, \ldots, u^m): \overline{\Omega}\to \R^m$, where $\Omega\subset \R^n$ is a bounded domain with smooth boundary $\partial \Omega$. The boundary data of $\Sigma$ is prescribed by a smooth function $\vec{\phi}: \partial \Omega\to \R^m$:
\[
\partial \Sigma= \{ (x, \vec{\phi}(x)) | x\in \partial \Omega       \}.
\]
The induced Riemannian metric on $\Sigma$ is
$
g_{ij}= \delta_{ij}- \partial_i \vec{u}\cdot \partial_j \vec{u}.
$
The maximal condition with the prescribed boundary data is equivalent to the \textbf{Dirichlet problem}:
\begin{equation}\label{Dirichlet}
\begin{cases}
\sum_{i,j=1}^n g^{ij} \partial_i \partial_j
u^\theta=0, \quad \theta=1, 2, \ldots m, \\
\vec{u}|_{ \partial \Omega}= \vec{\phi}.
\end{cases}
\end{equation}
The spacelike condition of the graph ensures this system is quasilinear elliptic. The boundary data is called \textbf{acausal} if $|\vec{\phi}(x)-\vec{\phi}(x')|< |x-x'|$ for any $x,x'\in \partial \Omega$. An equivalent way to write the Dirichlet problem is
\begin{equation}\label{Dirichlet2} 
\begin{cases}
\sum_{i,j} \partial_i( g^{ij}\sqrt{\det(g)}\partial_j u^\theta )=0, \quad \theta=1, \ldots, m. \\
\vec{u}|_{ \partial \Omega}= \vec{\phi}.
\end{cases}
\end{equation}
For smooth spacelike graphs, either (\ref{Dirichlet}) or (\ref{Dirichlet2}) implies
\[
\sum_i \partial_i( g^{ij}\sqrt{\det(g)})=0,\quad j=1,2,\ldots n.
\]
from which the equivalence of (\ref{Dirichlet}) and (\ref{Dirichlet2}) is clear.

Our main result is
\begin{thm}\label{Maintheorem}
	Given a bounded domain $\Omega$ with smooth boundary $\partial \Omega$, and given smooth acuasal boundary data $\vec{\phi}:\partial \Omega\to \R^m$, then there exists a unique smooth maximal graph $\vec{u}:\overline{\Omega}\to \R^m$ solving the Dirichlet problem (\ref{Dirichlet}), and morever it maximizes the volume functional among all spacelike graphs with the same boundary data.
\end{thm}

The system  (\ref{Dirichlet}) has strong formal analogies with the minimal surface system, the only difference in the setup being the signature of the ambient space. In the codimension one case, namely when $m=1$, the equations become scalar valued, and there is a rich theory for both the minimal and the maximal cases \cite{GilbargTrudinger}\cite{BartnikSimon}. In higher codimension, however, the minimal surface system is known to be poorly behaved by the striking results in \cite{Lawson}: even if we assume $\Omega$ is a bounded, smooth and convex domain, the Dirichlet problem can fail to have a solution; when the solution exists, it can be non-unique; solutions do not need to be volume minimising. The counterexamples in \cite{Lawson} require large gradient and involve a certain amount of nontrivial topology.

The principal interest of this paper lies in the contrast between the maximal and minimal graphs, both geometrically and analytically. This contrast is partially known to previous workers on related questions; to illustrate with a few elementary observations:
\begin{itemize}
\item	Fix the splitting $\R^{n,m}=\R^n \oplus \R^m$. Then any $n$-dimensional spacelike subspace of $\R^{n,m}$ is graphical over $\R^n$. So if we ignore the boundary, then any spacelike submanifold is locally an unramified covering over some open subset in the fixed $\R^n$. Intuitively, spacelike submanifolds are not far from being graphical over a fixed $\R^n$. 
\\
\item  Spacelike graphs have an \textbf{a priori gradient bound}, namely that for any unit vector $v=(v_1, \ldots v_n)\in \R^n$, 
\begin{equation}\label{apriorigradientbound}
|\sum_j v_j \partial_j \vec{u}| < 1.
\end{equation}
This already prevents the mechanisms of counterexamples in \cite{Lawson}.	When the domain $\Omega$ is convex, this gradient bound makes evident the necessity of the acausal boundary condition.
\\
\item   Maximal graphs have non-negative Ricci curvature (\cf \cite{Donaldson1}, and Lemma \ref{Riccipositive} below). This follows immediately from the Gauss equation. In contrast, minimal graphs have non-positive Ricci curvature, which offers less analytic control.

\end{itemize}
Another much deeper fact is the \textbf{Bernstein type theorem}, stating that maximal graphs over the entire $\R^n$ must be linear (the codimension $m=1$ case is due to \cite{Yau}, and for the higher codimension case see Theorem 4.2 in \cite{Jost}), in contrast with the minimal graphs, for which the Bernstein theorem is only satisfied for $n\leq 7$  even for $m=1$ (\cf survey in \cite{Yau}).

Theorem \ref{Maintheorem} can be seen as the generalisation of the codimension one case treated by R. Bartnik and L. Simon \cite{BartnikSimon}, and the proof strategy also follows theirs quite closely. The main issue is that tangent vectors on maximal graphs have the a priori possibility of approaching null directions, thereby destroying uniform ellipticity estimates for the quasilinear system (\ref{Dirichlet}). The core of this paper is to prevent this from happening through a barrier construction and a maximum principle argument.

\begin{rmk}
The same Dirichlet problem was also considered in \cite{Thorpe}, which relied on the much more restrictive hypothesis that the boundary data $\vec{\phi}$ is small in the $C^2$-norm.
\end{rmk}

\begin{rmk}
A particular consequence of Theorem \ref{Maintheorem} is that any given smooth acausal data $\vec{\phi}: \partial \Omega\to \R^m$ can be realised as the boundary of a spacelike graph over $\overline{\Omega}$. Does this fact have a more elementary proof without using PDE theory?
\end{rmk}

Theorem \ref{Maintheorem} is pertinent to $G_2$ geometry by the following result of Baraglia.

\begin{thm}
\cite{Baraglia} Let $B$ be a simply connected domain in $\R^3$. The data of torsion free $G_2$ manifolds $M$ with coassociative $T^4$-fibrations $M\to B$ such that the $T^4$ fibres are flat, are up to isomorphism in bijective correspondence with maximal immersions $f: B\to \R^{3,3}=H^2(T^4)$  (\ie the image $f(B)$ is a maximal submanifold in $\R^{3,3}$ modulo self intersection issues) up to translations.
\end{thm}

Donaldson \cite{Donaldson1} recently proposed that  maximal submanifolds with $(n,m)=(3,19)$ should arise from the adiabatic limiting description of $G_2$ submanifolds admitting a coassociative K3 fibration. Dirichlet problem for maximal submanifolds are tied to boundary value problem for $G_2$ manifolds; this application was the original motivation of this paper, which we leave for future investigation.

\begin{Notation}
The Lorentzian type inner product on $\R^{n,m}$ will be denoted $\langle, \rangle$. The Euclidean inner products on $\R^n$ and $\R^m$ are denoted by the dot product; thus in particular $\langle v, v'\rangle=-v\cdot v'$ on the $\R^m$ factor. The Levi-Civita connection is denoted by $\nabla$. The second fundamental form $A$ on the maximal graph $\Sigma$ is defined by $A(e_i, e_j)=(\nabla_{e_i} e_j)^\perp$ where $\{e_i\}$ is an orthonormal frame on $\Sigma$, and the mean curvature is $\vec{H}_{\Sigma}=\sum_i A(e_i, e_i)$. The Laplacian on $\Sigma$ follows the analysts' convention, namely $\Lap_{\Sigma} f= \text{div} \nabla f$.
\end{Notation}


\begin{Acknowledgement}
	The author is grateful to his PhD supervisor Simon Donaldson and co-supervisor Mark Haskins for their inspirations, Jason Lotay for discussions, and the Simons Center for hospitality. 
	
	This work was supported by the Engineering and Physical Sciences Research Council [EP/L015234/1], the EPSRC Centre for Doctoral Training in Geometry and Number Theory (The London School of Geometry and Number Theory), University College London. The author is also funded by Imperial College London for his PhD studies.
\end{Acknowledgement}

\section{Uniqueness of maximal graph}

The goal of this Section is to show

\begin{thm}\label{uniquenessmaximal}(\textbf{Uniqueness})
Let $\Sigma_1$, $\Sigma_2$ be two smooth maximal graphs over the smooth bounded domain $\Omega\subset \R^n$, with the same acausal boundary data $\partial \Sigma$. Then $\Sigma_1=\Sigma_2$, and morever it \textbf{maximizes the volume functional} among all smooth spacelike graphs with the same boundary data.
\end{thm}

We begin by observing

\begin{lem}\label{acausalboundary}
Let $\Sigma$ be any smooth space-like graph over the smooth domain $\Omega\subset \R^n$, with acausal boundary data. Then 
\[
|\vec{u}(x)-\vec{u}(x')|<|x-x'|, \forall x,x'\in \overline{\Omega}.
\]
\end{lem}

\begin{proof}
Join $x$ and $x'$ by the straight line segment $tx+(1-t)x'$. If the segment is contained in $\overline{\Omega}$, then the estimate follows from the gradient bound (\ref{apriorigradientbound}). If the segment leaves $\overline{\Omega}$ (which is possible since the domain may not be convex), then we subdivide the segment according to the the time it crosses the boundary, apply the acausal boundary condition for every sub-segment lying in the exterior domain, and use the triangle inequality to conclude.
\end{proof}

The key to the uniqueness theorem is the ability to represent $\Sigma_2$ as a section of the normal bundle of $\Sigma_1$ inside the Lorentzian type space $\R^{n,m}$. More precisely,

\begin{lem}
Let $\Sigma_1, \Sigma_2$ be two spacelike $n$-dimensional graphs with the same acausal boundary $\partial \Sigma$.
For every $p\in \Sigma_1$, there is a unique normal vector $\nu(p)\in \R^{n,m}$, such that $\nu(p)\perp T_p \Sigma_1$ and $p+\nu(p)\in \Sigma_2$. Morever $\nu(p)$ depends smoothly on $p$.
\end{lem}

\begin{proof}
Denote $(T_p\Sigma_1)^\perp$ as the orthogonal complement of the spacelike subspace $T_p\Sigma_1\subset \R^{n,m}$. We are required to find a unique intersection point of $\Sigma_2$ with the normal affine plane $p+(T_p \Sigma_1)^\perp$. As a preliminary observation, since  $\Sigma_2$ is spacelike of dimension $n$ while the normal plane is timelike of dimension $m$, the intersection must be transverse and of complementary dimension in $\R^{n,m}$. Morever this intersection has a sign: by the connectedness of the Grassmannian of spacelike $n$-planes in $\R^{n,m}$, we can consistently give orientations such that any  spacelike $n$-plane and any  timelike $m$-plane in $\R^{n,m}$ have intersection number 1.

When $p\in \partial \Sigma$ then this intersection point is just $p$ itself, namely $\nu(p)=0$. It is unique, because any $q\in \Sigma_2\setminus \{p\}$ must be spacelike separated from $p$ by the acausal condition, so $q$ cannot lie in the normal plane.

When $p$ is an interior point, then the normal plane cannot intersect $\partial \Sigma$ by the above argument. This means the linking number of $\partial \Sigma$ with the normal plane is well defined: it is just the degree of $[\partial \Sigma]\in H_{n-1}(  \R^{n,m}\setminus( p+ (T_p\Sigma_1)^\perp) )\simeq\Z$, or equivalently the count of intersection numbers of $\Sigma_i$ with the normal plane, for $i=1,2$. Clearly $\Sigma_1$ intersects the normal plane at a unique point $p$, so the linking number is 1, and by the positivity of intersection $\Sigma_2$ must intersect the normal plane transversely at 1 point.

Finally, to see $\nu$ is smooth, we can apply the implicit function theorem to the defining conditions
\[
\begin{cases}
\nu(p) \perp T_p\Sigma_1, \\
q=p+\nu(p)  \in \Sigma_2
\end{cases}
\]
and notice the nondegeneracy condition is precisely that the intersection of $T_q\Sigma_2$ with $(T_p\Sigma_1)^\perp$ is transverse.
\end{proof}

\begin{proof}
(Theorem \ref{uniquenessmaximal})
Let $\Sigma_1$, $\Sigma_2$ be any two spacelike graphs with the same boundary $\partial \Sigma$.
 We write $\Sigma_2=\{ q=p+\nu(p) | p\in \Sigma_1       \}.
$ Let $e_1, \ldots e_n$ be a pointwise orthonormal basis of $T_p \Sigma_1$, then a basis of tangent vectors to $\Sigma_2$ at $q$ is given by $e_i+ \nabla_{e_i} \nu$ where $\nabla$ denotes the Levi-Civita connection, hence the metric tensor on $\Sigma_2$ is described by
the positive definite matrix $Q_{ij}=\langle e_i+ \nabla_{e_i} \nu, e_j+ \nabla_{e_j} \nu\rangle$, and the volume element on $\Sigma_2$ is $\sqrt{\det(Q)} dvol_{\Sigma_1}$. We have
\[
\text{Vol}(\Sigma_2)\leq \int_{\Sigma_1} \sqrt{\det(Q)} dvol_{\Sigma_1},
\]
where the inequality signifies the a priori posibility that a point $q\in\Sigma_2$ can be represented by several $p\in \Sigma_1$.

We now decompose $\nabla_{e_j}\nu= ( \nabla_{e_j}\nu  )^{T\Sigma_1}+ ( \nabla_{e_j}\nu  )^{T\Sigma_1\perp}
$ into the parts parallel and perpendicular to $T\Sigma_1$. Then
\[
\begin{cases}
Q_{ij}=Q'_{ij}+ Q_{ij}'',  \\

 Q_{ij}'=\langle e_i+ (\nabla_{e_i} \nu)^{T\Sigma_1}, e_j+ (\nabla_{e_j} \nu)^{T\Sigma_1}\rangle, \\
 
 Q_{ij}''= \langle   ( \nabla_{e_i}\nu  )^{T\Sigma_1\perp},   ( \nabla_{e_j}\nu  )^{T\Sigma_1\perp} \rangle. \\
\end{cases}
\]
Observe $Q_{ij}''$ is negative semi-definite by the timelike nature of $(T_p\Sigma_1)^\perp$, so the matrix $(Q_{ij}')\geq (Q_{ij})>0$, whence $\det(Q)\leq \det(Q')$.

On the other hand $Q_{ij}'$ is intimately related to the second fundamental form of $\Sigma_1$ in the normal direction $\nu$:
\[
\langle \nabla_{e_i} \nu, e_k\rangle= -\langle \nu, \nabla_{e_i}e_k\rangle=   - A_\nu( e_i, e_k   ), 
\]
where $A_\nu$ is a symmetric matrix on $T_p\Sigma_1$. Thus $(\nabla_{e_i} \nu)^{T\Sigma_1}= -\sum_k  A_\nu( e_i, e_k) e_k$. We now demand that in the orthonormal frame $\{e_i\}$, the matrix $A_\nu$ is diagonal, with eigenvalues $\lambda_1, \ldots \lambda_n$. Then $Q_{ij}'=(1-\lambda_i)^2 \delta_{ij}$. By the positive definiteness of $Q_{ij}'$ we see $\lambda_i\neq 1$. We claim $\lambda_i<1$: this is because $\max_i \lambda_i$ is a continous function on $\Sigma_1$, but on $\partial \Sigma$ it takes value 0 since $A_\nu=0$ there. Thus arithmetic-geometric inequality implies
\[
\det(Q')^{1/2} = \prod_i (1-\lambda_i) \leq ( 1- \frac{1}{n} \Tr A_\nu    )^n= (1- \frac{\langle \vec{H}_{\Sigma_1}, \nu \rangle}{n}  )^n,
\]
where $\vec{H}_{\Sigma_1}$ denotes the mean curvature vector on $\Sigma_1$.

Combining the above discussions, 
\begin{equation}
\text{Vol}(\Sigma_2)\leq \int_{\Sigma_1} \sqrt{\det(Q)}dvol_{\Sigma_1} \leq \int_{\Sigma_1} \sqrt{\det(Q')}dvol_{\Sigma_1}\leq \int_{\Sigma_1}(1- \frac{\langle \vec{H}_{\Sigma_1}, \nu \rangle}{n}  )^n dvol_{\Sigma_1}.
\end{equation}
In particular, if $\Sigma_1$ is a maximal submanifold, namely when $\vec{H}_{\Sigma_1}=0$, then $\text{Vol}(\Sigma_2)\leq \text{Vol}(\Sigma_1)$, so $\Sigma_1$ is indeed a maximizer of the volume functional.

Reversing the roles of $\Sigma_1$ and $\Sigma_2$, we see that if both $\Sigma_1$ and $\Sigma_2$ are maximal submanifolds, then their volumes are equal. The conditions to achieve all the equalities in the estimates force $\nu$ to be a parallel vector in $\R^{n,m}$; since $\nu=0$ on $\partial \Sigma$, it must then vanish globally, thus $\Sigma_1=\Sigma_2$.
\end{proof}

The maximality of the volume functional also has an infinitesimal version. Let $\Sigma_1$ be a smooth maximal graph as before, then the first order variation of the volume functional is zero. The second order variation can be extracted from the above computation, by taking $|\nu|<<1$:
\[
\begin{split}
\text{Vol}(\Sigma_2)& =\int_{\Sigma_1} \prod_i (1-\lambda_i) (1- \frac{1}{2} \sum_j |(\nabla_{e_j} \nu)^{ T\Sigma_1, \perp  } |^2 ) dvol_{\Sigma_1} + O(\nu^3) \\
& = \int_{\Sigma_1}  (1- \frac{1}{2}\sum_j \lambda_j^2-     \frac{1}{2}\sum_j |(\nabla_{e_j} \nu)^{ T\Sigma_1, \perp  } |^2 ) dvol_{\Sigma_1} + O(\nu^3) \\
& = \int_{\Sigma_1}  (1- \frac{1}{2}\sum_{i,j} |\langle A(e_i, e_j), \nu\rangle|^2- \frac{1}{2} \sum_j |(\nabla_{e_j} \nu)^{ T\Sigma_1, \perp  } |^2 ) dvol_{\Sigma_1} + O(\nu^3) \\
& = \int_{\Sigma_1}  (1- \frac{1}{2}\sum_{j} | (\nabla_{e_j}\nu )^{T\Sigma_1} |^2-  \frac{1}{2} \sum_j |(\nabla_{e_j} \nu)^{ T\Sigma_1, \perp  } |^2 ) dvol_{\Sigma_1} + O(\nu^3) 
\end{split}
\]
Here the second equality uses $\lambda_i=O(\nu)$ and $\sum_i \lambda_i=0$. From this, the \textbf{second variation of volume functional} induced by a normal vector field $\nu$ (defined as $\frac{\partial^2}{\partial t^2 }|_{t=0} \text{Vol}(\Sigma_1+ t\nu  )$) is
\begin{equation}\label{secondvariationofvolume}
\begin{split}
\delta^2 \text{Vol}(\nu, \nu) = 
\int_{\Sigma_1}   -\sum_{j} | (\nabla_{e_j}\nu )^{T\Sigma_1} |^2-   \sum_j |(\nabla_{e_j} \nu)^{ T\Sigma_1, \perp  } |^2 ) dvol_{\Sigma_1} .
\end{split}
\end{equation}
We see the second variation at $\Sigma_1$ is \textbf{negative definite}. The non-degeneracy is because if $\delta^2 \text{Vol}(\nu, \nu)=0$, then $\nabla \nu=0$, but $\nu$ is assumed to vanish on the boundary, so must be zero identically. In contrast, for minimal surfaces, the standard second variation formula does not always enjoy the positive definite property.

\section{Barrier construction and gradient estimate}

\subsection{Comparison hypersurfaces}

In this Section we introduce a family of hypersurfaces in $\R^{n,m}$, which will later serve as barriers to achieve boundary gradient estimates. In the special case of $m=1$, these reduce to rotationally symmetric constant mean curvature hypersurfaces in $\R^{n,1}$, which can be found by solving an ODE (\cf \cite{BartnikSimon}).

 Write $\R^{n,m}=\R^n_x\oplus \R^m_y$, and fix $\xi\in \R^n$, $\eta\in \R^m$. We denote the Euclidean distances by $r=|x-\xi|$ and $w=|y-\eta|$. Given a curve $\Gamma=\{ w=f(r)   \}$ in the $w-r$ plane where $f$ is a smooth positive function to be specified, then we obtain by rotation a hypersurface $\tilde{\Gamma}=\{w=f(r)\}\subset \R^{n,m}$, which by construction is symmetric under $SO(n-1)\times SO(m-1)$, and is foliated by $S^{n-1}\times S^{m-1}$. We demand $|f'|<1$, so the hypersurface has a well defined normal vector field in $\R^{n,m}$:
 \[
 \vec{n}= \frac{1}{  \sqrt{1-f'^2} }( \frac{\partial}{\partial w}+ f' \frac{\partial}{\partial r}  ), \quad \vec{n}\cdot \vec{n}=-1.
 \]

The information of the second fundamental form of the hypersurface is encoded in $\nabla_v \vec{n}$ for tangent vectors $v$ to $\tilde{\Gamma}$ at the point $(x,y)\in \R^{n,m}$:
\begin{itemize}
\item When $v=v_1$ is tangent to $S^{n-1}\times \{y\}$, a sphere of radius $r$, then
\[
\nabla_v \vec{n}= \frac{f'}{ \sqrt{1-f'^2} } \nabla_v \frac{\partial}{\partial r}= \frac{f'}{ r\sqrt{1-f'^2} }v, \quad \langle \nabla_v \vec{n}, v \rangle = \frac{f' |v|^2}{ r\sqrt{1-f'^2} }. 
\] \\
\item When $v=v_2$ is tangent to $\{x\}\times S^{m-1}$, a sphere of radius $w=f(r)$, then
\[
\nabla_v \vec{n}= \frac{1}{ f\sqrt{1-f'^2} } v , \quad \langle \nabla_v \vec{n}, v \rangle = \frac{- |v|^2}{ f\sqrt{1-f'^2} }. 
\]
Here one needs to be careful that $v$ is timelike, and $\langle v, v \rangle=-|v|^2$. 
\\

\item When $v=v_3$ is tangent to the $\tilde{\Gamma}$ but orthogonal to $S^{n-1}\times S^{m-1}$, then $v\in \R^{n,m}$ is proportional to $\frac{\partial}{\partial r}+ f' \frac{\partial}{\partial w}$ which is spacelike, and
\[
\nabla_v \vec{n}= \frac{f''}{ (1-f'^2)^{3/2} } v , \quad \langle \nabla_v \vec{n}, v \rangle= \frac{f''}{ (1-f'^2)^{3/2} } \langle v , v \rangle.
\] 
\end{itemize}
In general we can decompose a tangent vector to $\tilde{\Gamma}$  into the 3 types, to write out the second fundamental form of $\tilde{\Gamma}$:
\begin{equation}\label{secondfundamentalformhypersurface}
v=v_1+v_2+v_3, \quad
-\langle \nabla_v \vec{n}, v \rangle= \frac{-f' }{ r\sqrt{1-f'^2} }|v_1|^2+ \frac{1 }{ f\sqrt{1-f'^2} } |v_2|^2 - \frac{f''}{ (1-f'^2)^{3/2} } \langle v_3 , v_3 \rangle.
\end{equation}

We now impose that the function $f$ satisfies the equation
\begin{equation}
r^{1-n} \frac{d}{dr} (  \frac{ r^{n-1} f'} { \sqrt{1-f'^2}  }       )= (n-1) \frac{f' }{ r\sqrt{1-f'^2} }+ \frac{f''}{ (1-f'^2)^{3/2} }=   \Lambda,
\end{equation}
for some constant $\Lambda$; in the special case where $m=1$, this means the hypersurface $\tilde{\Gamma}=\tilde{\Gamma}_{K,\Lambda}$ has constant mean curvature.
Upon integration,
\[
\frac{f'} { \sqrt{1-f'^2}  } = \frac{\Lambda}{n} r+ Kr^{1-n},
\]
so
\begin{equation}
f_{K, \Lambda}(r)= \int_0^r  \frac{K+ \frac{1}{n}\Lambda t^n}  { \sqrt{ t^{2n-2}+( K+ \frac{1}{n}\Lambda t^n  )^2 }   }dt.
\end{equation}
We restrict attention to the range of parameters $K>0$, $\Lambda\leq 0$, $0<r< (\frac{nK} {|\Lambda|})^{1/n}$ (when $\Lambda=0$, this just means $r>0$), so
\[
0<f'<1, \quad 0<f<r,
\]
  namely the hypersurface is constrained to lie within the spacelike-cone with apex $(\xi, \eta)$. At $r=0$, the hypersurface is singular, and is tangent to the light cone. As we increase $K\geq K_1>0$ while fixing $\Lambda$, then $f'_{K, \Lambda}$ and $f_{K,\Lambda}$ are increasing, so the 1-parameter family of hypersurfaces $\tilde{\Gamma}_{K,\Lambda}$ stay disjoint, and as $K\to +\infty$ they approach the light cone $\{ w=r\}$, sweeping out the region $\{ (x,y)\in \R^{n,m}| f_{K_1, \Lambda}(r)\leq w< r, \text{ and } r< (\frac{nK_1} {|\Lambda|})^{1/n} \}$.

\begin{rmk}
When $n=2, m=1, \Lambda=0$, then $f_{K,\Lambda}(r)= K \sinh^{-1}( \frac{r}{K} )$. The Euclidean signature analogue of the rotationally invariant minimal hypersurface is the catenoid, defined using the function $K\cosh^{-1}(\frac{r}{K})$ where $r\geq K$.
\end{rmk}

Let $\Pi$ be any n-dimensional spacelike subspace of the tangent space of $\tilde{\Gamma}_{K,\Lambda}$, then we can consider the mean curvature over $\Pi$, defined by tracing the second fundamental form over an orthonormal basis $\{e_i\}$ of $\Pi$,
\[
H_{\Pi}= -\sum_{e_i} \langle e_i, \nabla_{e_i} \vec{n}\rangle.
\]

\begin{lem}\label{HPiupperbound}
Any $n$-dimensional space-like $\Pi$ satisfies
$H_\Pi \geq -\Lambda$.
\end{lem}

\begin{proof}
Using the second fundamental form formula (\ref{secondfundamentalformhypersurface}) and the inequalities
\[
\frac{1}{f \sqrt{1-f'^2}  }>    \frac{f'}{ r \sqrt{1-f'^2}    }>0 \geq\Lambda> \frac{f''}{ (1-f'^2)^{3/2}   },
\]
we see that $-\langle v, \nabla_v \vec{n}\rangle \geq -\frac{f'}{ r \sqrt{1-f'^2}    }\langle v, v\rangle$ for any $v\in \Pi$, so any eigenvalue of the second fundamental form is at least $-\frac{f'}{ r \sqrt{1-f'^2}    }$. Morever, on the 1-dimensional subspace $\{v_1=0\}\cap \Pi$, we have $
-\langle v, \nabla_v \vec{n}\rangle \geq -\frac{f''}{  ({1-f'^2})^{3/2}    }\langle v, v\rangle,
$ so the largest eigenvalue of the second fundamental form of $\tilde{\Gamma}$ is at least $-\frac{f''}{  ({1-f'^2})^{3/2}    }$. Thus the trace $H_\Pi$ is at least $ -\frac{f''}{  ({1-f'^2})^{3/2}    }-(n-1) \frac{f'}{ r \sqrt{1-f'^2}    }=-\Lambda$. The only way to achieve equality is $\Pi= \{ v_2=0  \}\cap T_p \tilde{\Gamma}_{K,\Lambda}$.
\end{proof}

We now derive a \textbf{comparison principle}.

\begin{lem}\label{comparisonprinciple}
Let $\Sigma$ be a maximal graph over the smooth bounded domain $\Omega\subset \R^n$, with acausal boundary $\partial \Sigma$. Assume $\xi\in \R^n\setminus \overline{\Omega}$, $\eta\in \R^m$, $K_1>0$, $\Lambda<0$ are chosen such that 
\[
\text{dist}(x, \xi)< ( \frac{nK_1}{|\Lambda|}  )^{1/n}, \forall x\in \Omega, \quad \text{and }
\partial \Sigma\subset \{ w\leq f_{K_1,\Lambda} (r) \},
\]
then $\Sigma\subset \{ w\leq f_{K_1,\Lambda} (r) \}$.
\end{lem}

\begin{proof}
We begin by observing that, using  the idea of Lemma \ref{acausalboundary}, the graph $\Sigma$ must be contained in the spacelike-cone $\{ w<r \}$ with apex $(\xi, \eta)$. This implies $\Sigma\subset \{  w\leq f_{K, \Lambda} (r) \}  $ for some sufficiently large $K\geq K_1$ since these sublevel sets exhaust the spacelike cone; take $K$ to be minimal. If $K=K_1$ then the lemma is proved. We will assume $K>K_1$ and derive a contradiction.

By assumption there is an intersection point $p\in \Sigma \cap \tilde{\Gamma}_{K, \Lambda }$, and $p\notin \partial \Sigma$. Consider the function $h=w-f_{K,\Lambda}(r)$ on a neighbourhood of $p$. Since $h$ achieves maximum on $\Sigma$ at $p$, we see $T_p\Sigma\subset T_p \tilde{\Gamma}_{K, \Lambda}$, and
\[
0\geq \Lap_\Sigma h= \sum_{i=1 }^n \text{Hess}_{\R^{n,m}}(h)(e_i, e_i)+ dh(\vec{H}_{\Sigma} )= \sum_{i=1 }^n \text{Hess}_{\R^{n,m}}(h)(e_i, e_i),
\]
where $e_i$ is an orthonormal basis of $T_p\Sigma$, and the second equality uses the mean curvature zero condition. But for $\Pi= T_p\Sigma \subset T_p \tilde{\Gamma}_{K, \Lambda}$,
\[
\sum_{i=1 }^n \text{Hess}_{\R^{n,m}}(h)(e_i, e_i)= \langle \nabla_{e_i} ( - \sqrt{1-f'^2} \vec{n}  ) , e_i \rangle=\sqrt{1-f'^2} H_{\Pi},
\]
where the first equality uses  $\nabla_{\R^{n,m}} h= - \sqrt{1-f'^2} \vec{n}$, which comes from taking the dual of $dh=dw- f' dr$ with respect to the Lorentzian metric. Combining the above shows $H_{\Pi}\leq 0$, which contradicts Lemma \ref{HPiupperbound}.
\end{proof}



\subsection{Boundary gradient estimate}

We will now apply the comparison principle (Lemma \ref{comparisonprinciple}) to achieve a gradient bound, following the argument of Proposition 3.1 in \cite{BartnikSimon} quite closely.

\begin{prop}\label{Boundarygradientestimate}
(\textbf{Boundary gradient estimate}) Let $\Omega$ be a smooth bounded domain, and let $\vec{u}: \overline{\Omega}\to \R^m$ be a smooth solution to the Dirichlet problem with boundary data $\vec{\phi}: \partial \Omega\to \R^m$. Assume there exists a constant $0<\mu_0< 1$ with 
\begin{equation}\label{quantitativeacausality}
|\vec{\phi}(x)-\vec{\phi}(x')|\leq (1-\mu_0)|x-x'|, \quad \forall x, x'\in \partial \Omega, 
\end{equation}
and $\norm{\vec{\phi}}_{C^2(\partial \Omega) }\leq \kappa$. Then there is a constant $0<\mu<1$ depending only on $n, \Omega, \mu_0, \kappa$, such that at any boundary point $x_0\in \partial \Omega$, for any unit vector $v=(v_1, \ldots, v_n)\in \R^n$, we have the boundary gradient estimate
\[
|D_v \vec{u}|(x_0) = |\sum_j v_j \partial_j \vec{u}| (x_0)\leq 1-\mu.
\]
\end{prop}

\begin{proof}
(\cf Appendix in \cite{BartnikSimon})
Assume $x_0=0\in \partial \Omega$ and $e_n=(0,0\ldots, 1)$ is the inward pointing unit normal to $\partial \Omega$ at 0, so the tangential gradient operator is $D'=(  \partial_1, \ldots \partial_{n-1},0)$ at 0. Let $\vec{\theta}=(\theta^1, \ldots \theta^m)\in \R^m$ denote any given unit vector. Without loss of generality, the vector
\[
D'|_{0} \vec{\phi}\cdot \vec{\theta}= \sum_{\alpha=1}^m (D'|_0\phi^\alpha) \theta^\alpha= ae_1, \quad e_1=(1,0, \ldots, 0)\in \R^n,
\]
where we recall $\vec{\phi}$ is the boundary data of $\vec{u}: \R^n\to \R^m$. By assumption (\ref{quantitativeacausality}) $|a|=|D'|_0\vec{\phi}\cdot \vec{\theta}|\leq 1-\mu_0<1$.

Our next aim is to  construct suitable barriers. Fix $\Lambda<0$. Let $K=\epsilon^{-1}$ for some small number $\epsilon>0$ to be specified. Since $\Omega$ is a bounded domain, we can always ensure
\[
\text{diam}(\Omega)<< (nK/|\Lambda|  )^{1/n}
\]
for small $\epsilon$.
Choose parameters $\xi=\xi_\epsilon=\epsilon(-b, 0, \ldots, 0, -1)/\sqrt{1+b^2}\in \R^n$, with $b=b_\epsilon$ to be determined, and $\eta=\vec{\phi}(0)- f_{K, \Lambda}(\epsilon)\vec{\theta}\in \R^m$, so the hypersurface $\tilde{\Gamma}_{\Lambda, K}$ with apex at $(\xi, \eta)\in \R^{n,m}$ passes through $(0, \vec{\phi}(0))\in \R^{n,m}$. To specify $b$, we regard $f_{K, \Lambda}(r)=f_{K,\Lambda}( |x-\xi|  )$ and $w=|y-\eta|=|\vec{u}-\eta|$  as functions on $\overline{\Omega}$. The tangential derivatives at 0 are demanded to satisfy
\begin{equation}\label{barrierconstructiongradientrequirement}
D'|_0 (w^2)= D'|_0 (f_{K,\Lambda}^2).
\end{equation}
We compute
\[
D'|_0 (w^2)= 2( D'|_0\vec{u}   ) \cdot (\vec{u}(0)-\eta)= 2 ( D'|_0\vec{\phi}   ) \cdot (\vec{\phi}(0)-\eta)= 2 f_{K,\Lambda}(\epsilon) D'|_0 \vec{\phi}\cdot \vec{\theta},
\]
\[
D'|_0 (f^2_{K, \Lambda})= 2 f_{K,\Lambda}(\epsilon) \frac{1+ \Lambda\epsilon^{n+1}/n}{  \sqrt{ \epsilon^{2n}+ (1+ \Lambda \epsilon^{n+1}/n)^2      } } \frac{b}{\sqrt{b^2+1} } e_1,
\]
so the condition (\ref{barrierconstructiongradientrequirement}) translates into
\[
\frac{1+ \Lambda\epsilon^{n+1}/n}{  \sqrt{ \epsilon^{2n}+ (1+ \Lambda \epsilon^{n+1}/n)^2      } } \frac{b}{\sqrt{b^2+1} }=a,
\] 
which determines $b$ as long as $\epsilon$ is sufficiently small.

\textbf{Claim}:  for sufficiently small $\epsilon=\epsilon(n, \Omega, \mu_0, \kappa, \Lambda)>0$, the boundary $\partial \Sigma\subset \{ w\leq f_{K,\Lambda}(r)   \}$.

\begin{itemize}
\item When $x\in \partial \Omega$ is sufficiently far away from the origin, we compare
\[
w(x)=|\vec{\phi}(x)- \eta|\leq |\vec{\phi}(x)- \vec{\phi}(0)|+ |\vec{\phi}(0)-\eta| \leq (1-\mu_0) |x|+ f_{K,\Lambda}(\epsilon)\leq (1-\mu_0) |x|+ \epsilon, 
\]
with
$
f_{K, \Lambda}(|x-\xi|)\geq |x-\xi|-C\epsilon \geq |x|-|\xi|-C\epsilon\geq |x|-C\epsilon,
$
where the constants depend only on $n, \Omega, \Lambda$,
to see that 
\[
w(x)\leq f_{K, \Lambda}(|x-\xi|),\quad  \forall |x|\geq C_1(n, \Omega, \mu_0, \Lambda )\epsilon .
\]

\item  When $x\in \partial \Omega$ is very close to the origin, namely $|x|< C_1(n, \Omega, \mu_0, \Lambda )\epsilon$, the Claim is local in nature. For simplicity of presentation we pretend the boundary portion $\partial \Omega \cap \{|x|< C_1 \epsilon \}$ is flat, namely it is a coordinate open set in $\{x_n=0\}$; the general case of smooth domain is no more difficult. We compute
\[
\partial_i f_{K,\Lambda}= f'_{K,\Lambda}(|x-\xi|) \frac{x_i-\xi_i}{|x-\xi|  },  \quad i=1, 2, \ldots , n-1,
\]
where \[
f'_{K,\Lambda}(t)= \frac{ 1+\epsilon \Lambda t^n/n}{\sqrt{\epsilon^2 t^{2n-2} +( 1+ \epsilon \Lambda t^n/n )^2  }  }=1-O(\epsilon),
\]
and 
\[
\partial_i \partial_j f_{K, \Lambda}= \frac{f'_{K,\Lambda}(|x-\xi|)}{|x-\xi|} [   \delta_{ij}- \frac{ (x_i-\xi_i)(x_j-\xi_j)}{ |x-\xi|^2  }  ] -O(\epsilon), \quad i, j=1, 2, \ldots, n-1.
\]
Combining this with $f_{K, \Lambda}=|x-\xi|(1-O(\epsilon))$, we get
\[
\partial_i \partial_j (f^2_{K,\Lambda})=  2 \delta_{ij}-O(\epsilon), \quad i,j=1, 2, \ldots, n-1.
\]
But on $\partial \Omega$
\[
\begin{split}
\partial_i \partial_j (w^2)&= \partial_i \partial_j |\vec{\phi}-\eta|^2 = \partial_i \partial_j |\vec{\phi}- \vec{\phi}(0)|^2 +O(\epsilon) \\
&=2 \partial_i \vec{\phi}\cdot \partial_j \vec{\phi}+ 2(\vec{\phi}-\vec{\phi}(0))\cdot \partial_i \partial_j \vec{\phi}+O(\epsilon) \\
&\leq  2 \partial_i \vec{\phi}\cdot \partial_j \vec{\phi} + 2\kappa |\vec{\phi}-\vec{\phi}(0)|+O(\epsilon) \\
&\leq 2\partial_i \vec{\phi}\cdot \partial_j \vec{\phi}+ C\epsilon\leq 2(1-\mu_0)^2 \delta_{ij}+ C\epsilon,
\end{split}
\]
where we used $\norm{\vec{\phi}}_{C^2(\partial \Omega)}\leq \kappa$ and the fact that the positive semidefinite matrix $(\partial_i \vec{\phi}\cdot \partial_j \vec{\phi})$ is dominated by $ (1-\mu_0)^2(\delta_{ij})$, which is a consequence of (\ref{quantitativeacausality}). Thus by choosing $\epsilon<<1$ depending on the given constants, we can assume for $|x|< C_1\epsilon$ that the Hessian matrices satisfy the inequality
\[
(\partial_i\partial_j (w^2)) \leq 2(1-\mu_0/2)^2 \delta_{ij}  \leq ( \partial_i \partial_j (f^2_{K, \Lambda})  ).
\]
But by construction the functions $w^2$ and $f^2_{K, \Lambda}$ are tangent to first order at the origin (\cf (\ref{barrierconstructiongradientrequirement})), so the above convexity property implies $w^2\leq f^2_{K, \Lambda}$ for $|x|< C_1\epsilon$, hence the Claim is proved.
\end{itemize}

Now it follows immediately from the Claim and Proposition \ref{comparisonprinciple} that the maximal graph $\Sigma\subset \{  w\leq f_{K,\Lambda}(r) \}$, namely
\[
|\vec{u}(x)- \eta|\leq f_{K,\Lambda}(|x-\xi |), \quad \forall x\in \overline{\Omega}.
\]
In particular, using that $(0, \vec{u}(0))\in \tilde{\Gamma}_{K,\Lambda}$, we have for $v=(v_1, \ldots, v_n)\in \R^n$ with $v_n>0$ and $|v|=1$,
\[
\frac{1}{t}(|\vec{u}(t v) -\eta|-|\vec{u}(0)|) \leq \frac{1}{t}( f_{K, \Lambda}(|tv-\xi|)- f_{K, \Lambda} (|\xi|)), \quad\forall 0< t<<1
\]
and take the limit as $t\to 0$, we have an estimate on the directional derivative of $\vec{u}$ at 0,
\[
D_v \vec{u}\cdot \vec{\theta}= \sum_i v_i \partial_i \vec{u}\cdot \vec{\theta} \leq  f'_{K, \Lambda}(|\xi|) \frac{-v\cdot \xi}{ |\xi|  } \leq f'_{K, \Lambda}(\epsilon) \leq 1-\mu<1,
\]
where $0<\mu<1$ is a constant depending only on $ n, \Omega, \mu_0, \kappa, \Lambda$. Since $\mu$ does not depend on the direction $\vec{\theta}$, in fact
\[
|D_v \vec{u} |= | \sum_i v_i \partial_i \vec{u}|   \leq 1-\mu.
\]
Taking the limit as $v_n\to 0$ we can allow $v_n\geq 0$, and for unit vectors $v\in \R^n$ with $v_n<0$ we have $|D_{-v} \vec{u}|=|D_v \vec{u}|$, so $|D_v \vec{u}|\leq 1-\mu$ holds uniformly for all unit $v\in \R^n$ at all boundary points.
\end{proof}

\subsection{Uniform ellipticity}

It is straightforward to obtain a uniform ellipticity estimate on $\vec{u}$ from the boundary gradient estimate, due to the presence of a favourable maximum principle.

\begin{lem}(\cf Proposition 7 in \cite{Donaldson1})\label{Riccipositive}
Let $\Sigma$ be a maximal submanifold of $\R^{n,m}$. Then the Ricci curvature of the induced metric is non-negative.
\end{lem}

\begin{proof}
Let  $e_1, \ldots e_n$ be an orthonormal frame of $\Sigma$. Let $A(e_i, e_j)=(\nabla_{e_i} e_j)^\perp$ denote the second fundamental form of $\Sigma$. The Gauss equation for the submanifold $\Sigma\subset \R^{n,m}$ expresses the intrinsic Ricci tensor as
\[
\text{Ric}(X,Y)=\sum_j \langle A(e_j, e_j), A(X, Y)\rangle- \langle A(X, e_j), A(Y, e_j) \rangle.
\]
But the mean curvature $\sum_j A(e_j, e_j)=0$, so only the second term remains. Since the normal vectors $A(X, e_i)$ is timelike or zero, we see $\text{Ric}(X, X)\geq 0$ as required.
\end{proof}

\begin{prop}
(\textbf{Uniform ellipticity}) In the setting of Proposition \ref{Boundarygradientestimate}, the metric tensor $(g^{ij})=(g_{ij})^{-1}$ satisfies the estimate
\begin{equation}\label{Uniformellipticity}
\sum_i g^{ii} \leq \frac{n}{ \mu(2-\mu)  }.
\end{equation}
\end{prop}

\begin{proof}
For $\theta=1, \ldots, m$, the component functions $u^\theta$
are harmonic on $\Sigma$, by the maximal submanifold condition. So we can apply a well known Bochner formula (\cf eg. \cite{Cheeger} equation 1.10) to get
\[
\frac{1}{2} \Lap_\Sigma |\nabla u^\theta|^2= |\text{Hess}_{\Sigma} (u^\theta) |^2 + \text{Ric}( \nabla u^\theta,  \nabla u^\theta  ) \geq 0.
\]
Summing  over $\theta=1,\ldots, m$, 
\[
\Lap_\Sigma (\sum_{\theta=1}^{m} |\nabla u^\theta|^2) \geq 0.
\]
Thus $\sum_{\theta=1}^{m}|\nabla u^\theta|^2=\sum_{i,j, \theta} g^{ij} \partial_i u^\theta \partial_j u^\theta= \sum_{i,j} g^{ij} \partial_i \vec{u} \cdot \partial_j \vec{u}  $ achieves maximum on $\partial \Omega$. But from the boundary gradient estimate Proposition \ref{Boundarygradientestimate}, 
the metric tensor $g_{ij}= \delta_{ij}- \partial_i \vec{u} \cdot \partial_j \vec{u}$ is uniformly equivalent to $\delta_{ij}$ on $\partial \Omega$:
\[
(\delta_{ij}) \geq  (g_{ij}) \geq (1-(1-\mu)^2) (\delta_{ij})= \mu(2-\mu) (\delta_{ij}),
\]
where $\mu$ is as in Proposition \ref{Boundarygradientestimate}.
Thus on $\partial \Omega$,
 \[
\sum_{i,j} g^{ij} \partial_i \vec{u} \cdot \partial_j \vec{u} = \sum_{i,j} (g^{ij}\delta_{ij}-  g^{ij}g_{ij})=\sum_i g^{ii}-n\leq \frac{n}{ \mu(2-\mu) }-n.
\]
Therefore at any $x\in \overline{\Omega}$,
\[
\sum_{\theta=1}^{m}|\nabla u^\theta|^2=\sum_{i,j} g^{ij} \partial_i \vec{u}\cdot \partial_j\vec{u} \leq \frac{n}{\mu(2-\mu)  }-n,
\]
or equivalently $\sum_i g^{ii}= n+ \sum_{i,j} g^{ij} \partial_i \vec{u}\cdot \partial_j\vec{u} \leq \frac{n}{\mu(2-\mu)  }$.
\end{proof}

\begin{rmk}
Once we have the boundary gradient estimate, uniform ellipticity can also be derived from the fact that $\Lap_{\Sigma} \log \det(g_{ij}) \leq 0$, as in \cite{Thorpe}.
\end{rmk}

\section{$C^{1,\alpha}$-estimate and existence of solution}

\subsection{$C^{1,\alpha}$-estimate}

As explained in \cite{Thorpe} there is a small logical gap between uniform ellipticity estimate and higher order estimates due to the vector valued nature of $\vec{u}$. This will be bridged by a $C^{1,\alpha}$-estimate.
For $x\in \Omega$, denote $d_x=\text{dist}(x, \partial \Omega)$. We write the following semi-norms for a $C^2$-function $\bar{u}$ on $\Omega'\subset \overline{\Omega}$:
\[
[D \bar{u}]_{\alpha, \Omega'}= \sup_{x,x' \in \Omega'} \sum_i \frac{|\partial_i \bar{u}(x)- \partial_i \bar{u}(x) | }{|x-x'|^\alpha},\quad [D^2 \bar{u}]_{0,\Omega'}=\sup_{\Omega'} \sum_{i,j} |\partial_i\partial_j \bar{u} |.
\]

\begin{lem}\label{interiorC2}
In the setup of Proposition \ref{Boundarygradientestimate},
the second derviatives satisfy the interior estimate
\[
|\partial_i \partial_j \vec{u}(x)| \leq C(n, \Omega, \mu_0, \kappa) d_x^{-1}, \quad i,j=1, \ldots, n.
\]
\end{lem}

\begin{proof}
From the uniform ellipticity estimate, we see that the geodesic distance between any two points $(x, \vec{u}(x))$ and $(x', \vec{u}(x'))$ on $\Sigma$ is bounded from below by $C^{-1}|x-x'|$. In particular, for an interior point $x\in{\Omega}$ the geodesic ball with radius $C^{-1}d_x$ is contained in the interior of $\Sigma$. According to  (2.9) in \cite{Jost} (\cf also Section 4 in \cite{Yau}), in the geodesic ball with halved radius we have a second fundamantal form bound
\[
|A| \leq C d_x^{-1},
\]	
where the constant changes from line to line.
The reader can understand this estimate as an effective version of the Bernstein theorem for maximal submanifolds mentioned in the introduction.

The norm on $A$ is defined in terms of the natural induced metrics on the tangent and normal bundles of $\Sigma$. To convert this into coordinate expressions,
\[
|A|^2= \sum_{i,j,k,l=1}^n \sum_{\theta=1}^m  \frac{1}{ 1-|\sum_{i=1}^n \partial_i u^\theta|^2  } g^{ik}g^{jl}(\partial_i \partial_j u^\theta) (\partial_k \partial_l u^\theta)\geq \sum_{i,j=1}^n \sum_{\theta=1}^{m} |\partial_i \partial_j u^\theta|^2,
\]
from which the second order estimate follows.	
\end{proof}

We extend the boundary data $\vec{\phi}: \partial \Omega\to \R^m$ to a smooth vector valued function $\vec{\bar{\phi}}=(\bar{\phi}^1, \ldots ,\bar{\phi}^m)$ on $\overline{\Omega}$, whose $C^2$-norm is bounded in terms of only $n, \Omega, \kappa$. For $\theta=1,2,\ldots,m$, the functions $u^\theta- \bar{\phi}^\theta$ have zero boundary value, and satisfy the uniformly elliptic linear PDE
\[
\sum_{i,j} g^{ij} \partial_i \partial_j (u^\theta-\bar{\phi}^\theta )=- \sum_{i,j} g^{ij} \partial_i \partial_j \bar{\phi}^\theta, 
\]
where the coefficient matrix $g^{ij}$ is regarded as fixed, the RHS is $L^\infty$ controlled, and $u^\theta- \bar{\phi}^\theta$ is a priori controlled in $C^1$ norm.
We now apply Krylov's boundary gradient H\"older estimate (\cf Theorem 9.31 in \cite{GilbargTrudinger}) to obtain

\begin{lem}\label{KrylovboundarygradientHolder}
Suppose $\partial \Omega$ has a flat portion, namely $\overline{\Omega}\cap \{ |x|< R_0  \}=B^+_{R_0}=\{ |x|<  R_0, x_n \geq 0  \}$. Then for any $R\leq R_0$,
\[
\text{osc}_{B^+_R } \frac{u^\theta- \bar{\phi}^\theta  }{x_n } \leq C (\frac{R}{R_0})^{\bar{\alpha}},
\]
where $C$ and $\bar{\alpha}$ only depend on $n, \Omega, \mu_0, \kappa$.
\end{lem}

\begin{rmk}
When $\partial \Omega$ is not flat, then we can straighten the boundary by a local coordinate change. The role of $x_n$ is then replaced by any local boundary defining function. The upshot is that we can without loss of generality pretend the boundary is locally flat.
\end{rmk}

\begin{prop}
(\textbf{Global $C^{1, \alpha}$-estimate}) In the setting of Proposition \ref{Boundarygradientestimate}, we have a uniform bound
\begin{equation}\label{gradientHolder1}
[Du^\theta]_{\alpha, \overline{\Omega}}\leq C, \quad \theta=1,2,\ldots m,
\end{equation}
where $\alpha$ and $C$ depend only on $n, \Omega, \kappa, \mu_0$.
\end{prop}

\begin{proof}
The proof is essentially an interpolation argument (\cf Problem 13.1 in \cite{GilbargTrudinger}). Given Lemma \ref{interiorC2}, it is enough to pretend $\overline{\Omega}\cap \{|x|<R_0\}=B_{R_0}^+$ and prove gradient H\"older estimate in $B_{R_0/4}^+$.

We start with an interpolation inequality (\cf Lemma 6.32 in \cite{GilbargTrudinger}): for any interior point $x\in B_{R_0/4}^+$, any $C^2$-function $\bar{u}$, and any $0<\epsilon\leq 1$,
\[
d_x^{1+\alpha} [ D\bar{u} ]_{\alpha, \{ x': |x'-x|< \frac{1}{2} d_{x}  \}} \leq C( \epsilon^{1-\alpha} d_x^2 [D^2 \bar{u}]_{ 0, \{ x': |x'-x|< \frac{2}{3} d_x   \}  } + \epsilon^{-1-\alpha} |\bar{u}|_{L^\infty( \{ x': |x'-x|< \frac{2}{3} d_x   \}   )}         ).
\]
We choose $\epsilon= d_x^{ \frac{\alpha}{1-\alpha}  }$, and assuming
$[D^2 \bar{u}]_{ 0, \{ x': |x'-x|< \frac{2}{3} d_x   \}  }\leq C d_x^{-1}$ we get
\[
[ D\bar{u} ]_{\alpha, \{ x': |x'-x|< \frac{1}{2} d_{x}  \}}\leq C(1+  d_x^{-\frac{1+\alpha}{1-\alpha} } |\bar{u}|_{L^\infty( \{ x': |x'-x|< \frac{2}{3} d_x   \}   )}    ). 
\]
In particular, this estimate applies to
\[
\bar{u}(x')= u^\theta(x')- \bar{\phi}^\theta(x')- x_n' \partial_n |_{(x_1, \ldots x_{n-1},0)} (u^\theta-\bar{\phi}^\theta),
\]
thanks to the interior $C^2$-estimate Lemma \ref{interiorC2}. The function $\bar{u}$ has the additional feature that its derivatives vanish at the boundary point $(x_1, \ldots, x_{n-1}, 0)$. We apply Lemma \ref{KrylovboundarygradientHolder} to the upper-half-balls  centred at $(x_1, \ldots, x_{n-1}, 0)$, to see that
\[
|\bar{u}|_{L^\infty( \{ x': |x'-x|< \frac{2}{3} d_x   \}   )}\leq C x_n d_x^{\bar{\alpha}}  \leq C d_x^{ 1+ \bar{\alpha}  } . 
\]
Choosing $0<\alpha<1$ with $\frac{1+\alpha}{1-\alpha}\leq 1+\bar{\alpha}$, we obtain $[ D\bar{u} ]_{\alpha, \{ x': |x'-x|< \frac{1}{2} d_{x}  \}}\leq C$, which combined with the $C^2$-bound on $\bar{\phi}^\theta$ gives
\begin{equation}\label{gradientHolder2}
[ D u^\theta ]_{\alpha, \{ x': |x'-x|< \frac{1}{2} d_{x}  \}}\leq C.
\end{equation}
This is almost our goal, except that the radius $\frac{1}{2}d_x$ degenerates near the boundary. To overcome this, we observe (\ref{gradientHolder2}) implies
\[
\begin{split}
&|D u^\theta(x)-D u^\theta(x_1, \ldots, x_{n-1}, 0)| \\
&\leq \sum_{k\geq 0} |Du^\theta(x_1, \ldots x_{n-1}, x_n 2^{-k})- Du^\theta(  (x_1, \ldots x_{n-1}, x_n 2^{-k-1}   )| \\
&\leq C( \sum_k (x_n 2^{-k})^\alpha  ) \leq C x_n^\alpha.
\end{split}
\]
Lemma \ref{KrylovboundarygradientHolder} gives that for $x, x'\in B_{R_0/4}^+$,
\[
|D u^\theta(x_1', \ldots, x_{n-1}', 0)-D u^\theta(x_1, \ldots, x_{n-1}, 0)| \leq C|x-x'|^{\bar{\alpha}} \leq C|x-x'|^{\alpha}.
\]
Thus by triangle inequality we can achieve a H\"older bound 
\[
|D u^\theta(x)-D u^\theta(x')| \leq  C|x-x'|^{\alpha}
\]
when $\min (x_n, x'_n)\geq \frac{1}{2}|x-x'|$ and $x,x'\in B_{R_0/4}^+$. This complements (\ref{gradientHolder2}) to imply the result.
\end{proof}

Once we have achieved $C^{1,\alpha}$ estimate on $\vec{u}$, Schauder theory applied to (\ref{Dirichlet}) allows us to estimate all higher order derivatives.

\begin{prop}\label{Higherorderbounds}
(\textbf{Higher order estimate}) In the setting of Proposition \ref{Boundarygradientestimate}. For $k\geq 2$, given a H\"older bound on the boundary data $\norm{\vec{\phi}}_{C^{k,\alpha}(\partial \Omega)}\leq \kappa(k, \alpha)$, then there is a global H\"older bound
\[
\norm{\vec{u}}_{ C^{k,\alpha}(\overline{\Omega})  } \leq C(n, \Omega, \mu_0, \kappa, k, \alpha, \kappa(k,\alpha)).
\]
\end{prop}

\subsection{Continuity method and existence theorem}

We are now in the position to prove

\begin{thm}
(\textbf{Existence}) Given a smooth bounded domain $\Omega$ and acausal boundary data $\vec{\phi}: \partial\Omega\to \R^m$, there exists a smooth $\vec{u}: \overline{\Omega}\to \R^m$ solving the Dirichlet problem (\ref{Dirichlet}), or equivalently (\ref{Dirichlet2}).
\end{thm}

We use the \textbf{continuity method}, namely we consider the 1-parameter family of Dirichlet problems, where the maximal graph equation is the same as in (\ref{Dirichlet2}), but the boundary data is changed to $t\vec{\phi}$ for $0\leq t\leq 1$. For $t=0$, then $\vec{u}=0$ provides the trivial solution.  As usual, we need to show the set 
$S=\{ t\in [0,1] | \text{ There is a $C^{2,\alpha}$ solution with boundary $t\vec{\phi}$ }       \}$ is open and closed. We remark that standard Schauder theory implies that $C^{2,\alpha}$ solutions are automatically smooth.

Notice the boundary data $t\vec{\phi}$ is also acausal, and satisfy better bounds compared to $\vec{\phi}$. Applying Proposition \ref{Higherorderbounds} easily shows that $S$ is \textbf{closed}.

To show the \textbf{openness} statement, by the Implicit Function Theorem it suffices to study the linearised  operator $\mathcal{L}$ of $ \vec{u}\mapsto \sum \partial_i ( g^{ij} \partial_j \vec{u}      )$ at a solution of the Dirichlet problem with boundary data $t\vec{\phi}$. It is instructive to relate this to the variation of the volume functional on the set of spacelike graphs with the fixed boundary data:
\[
\text{Vol}( \vec{u} )= \int_{\Omega} \sqrt{\det(g)} dx_1\ldots dx_n.
\]
Given a first variation $\vec{w}: \overline{\Omega}\to \R^m$ to $\vec{u}$, then (using Einstein summation convention)
\[
\delta g_{ij}= -\partial_i \vec{u} \cdot \partial_j\vec{w}_j - \partial_j \vec{u} \cdot \partial_i\vec{w} , 
\quad
\delta \sqrt{\det(g)}=\frac{1}{2} \sqrt{\det(g)} g^{ij} \delta g_{ij}= - \sqrt{\det(g)} g^{ij} \partial_i \vec{u}\cdot \partial_j\vec{w},
\]
so the first variation of volume is
\[
\delta \text{Vol}_{\vec{u} } (\vec{w})= -\int_{\Omega} g^{ij}\sqrt{\det(g)} \partial_i\vec{u} \cdot \partial_j\vec{w}= \int_{\Omega} \partial_j (g^{ij} \sqrt{\det(g)} \partial_i\vec{u})\cdot \vec{w}.
\]
The linearisation $\mathcal{L}$ is seen to be the Hessian of the volume functional (\ie $\mathcal{L}$ is a version of the Jacobi operator):
\[
\delta^2 \text{Vol}_{\vec{u} }(w, w')= \int_{\Omega} (\mathcal{L} \vec{w} )\cdot \vec{w'}= \int_{\Omega} (\mathcal{L} \vec{w'} )\cdot \vec{w}.
\]
Now, if we denote $\nu$ as the projection of the vector field $\vec{w}$ to the normal bundle of $\Sigma$, then the above expression is seen as the same thing as (\ref{secondvariationofvolume}):
\[
\delta^2 \text{Vol}_{\vec{u} }(\vec{w}, \vec{w})= \delta^2 \text{Vol} (\nu, \nu),
\]
which is negative definite, so the kernel of $\mathcal{L}$ must be zero. The formally self-adjoint linearised operator $\mathcal{L}: C^{2,\alpha}_0(\overline{\Omega})\to C^{0,\alpha}(\overline{\Omega})$ (the subscript $0$ means zero boundary data) is then seen to be a bijection, which implies the openness of $S$.

\begin{rmk}
Once we have achieved the higher order estimate Proposition \ref{Higherorderbounds}, the existence theorem can be proved by other methods, such as the Leray-Schauder approach used in \cite{Thorpe}.
\end{rmk}


\end{document}